\documentclass{amsart}
\renewcommand{\baselinestretch}{1.5}
\usepackage{graphicx}
\usepackage[mathscr]{eucal}
\usepackage{amsfonts}
\usepackage{enumerate}
\usepackage{multicol}
\usepackage{array}
\usepackage{tikz}
\usepackage{mytab}
\usepackage{stackengine}

\theoremstyle{plain}
\newtheorem{thm}{Donotwrite}[section]

\newtheorem{definition}[thm]{Definition}
\newtheorem{theorem}[thm]{Theorem}

\newtheorem{lemma}[thm]{Lemma}

\theoremstyle{definition}
\newtheorem{example}[thm]{Example}

\newtheorem{remark}[thm]{Remark}

\numberwithin{equation}{section}

\newcommand{\nc}{\newcommand}

\nc{\op}{\oplus} \nc{\pv}{P^{\vee}}

\nc{\B}{\mathbf{B}} \nc{\V}{\mathbf{V}} 
\nc{\nbinom}[2]{\genfrac{}{}{0pt}{1}{{#1}}{{#2}}}
\nc{\qbinom}[2]{\left[\genfrac{}{}{0pt}{1}{{#1}}{{#2}}\right]}
\nc{\ft}{\tilde{f}} 
\nc{\et}{\tilde{e}} 
\nc{\Y}{\mathbf{Y}}
\nc{\T}{\mathbf{T}}
\nc{\R}{\mathbf{R}}
\nc{\C}{\mathbf{C}}
\nc{\D}{\mathbf{D}}
\nc{\ra}{\rightarrow} 
\nc{\vep}{\varepsilon} 
\nc{\vp}{\varphi}
\nc{\h}{\mathfrak{h}} 
\nc{\oP}{\overline{P}}
\nc{\Fit}{\tilde{F}_i} 
\nc{\Eit}{\tilde{E}_i}
\nc{\fit}{\tilde{f}_i} 
\nc{\eit}{\tilde{e}_i}
\nc{\mf}{\mathfrak}
\nc{\ds}{\displaystyle}
\nc{\oa}{a'}
\nc{\ob}{b'}
\nc{\mc}{\mathcal}
\nc{\bsx}{\boldsymbol{x}}
\nc{\imin}{i_{min}}
\nc{\imax}{i_{max}}

\nc{\dmin}{d_{min}}
\nc{\dmax}{d_{max}}

\allowdisplaybreaks

\begin{document}

\title[Multiplicities of some maximal dominant weights
of the $\widehat{s\ell}(n)$-modules $V(k\Lambda_0)$]{Multiplicities of some maximal dominant weights \\
of the $\widehat{s\ell}(n)$-modules $V(k\Lambda_0)$}

\author{Rebecca L. Jayne}
\address{Hampden-Sydney College, Hampden-Sydney, VA 23943}
\email{rjayne@hsc.edu} 

\author{Kailash C. Misra}
\address{Department of Mathematics, North Carolina State University,  Raleigh,  NC 27695-8205}
\email{misra@ncsu.edu}

\thanks{KCM: partially supported by Simons Foundation grant \#  636482}

\begin{abstract} For $n \geq 2$ consider the affine Lie algebra $\widehat{s\ell}(n)$ with simple roots $\{\alpha_i \mid 0 \leq i \leq n-1\}$.  Let $V(k\Lambda_0), \, k \in \mathbb{Z}_{\geq 1}$ denote the integrable highest weight $\widehat{s\ell}(n)$-module with highest weight $k\Lambda_0$. It is known that there are finitely many maximal dominant weights of $V(k\Lambda_0)$.
Using the crystal base realization of $V(k\Lambda_0)$ and lattice path combinatorics we determine the multiplicities of a large set of maximal dominant weights of the form $k\Lambda_0 - \lambda^\ell_{a,b}$ where $ \lambda^\ell_{a,b} = \ell\alpha_0 + (\ell-b)\alpha_1 + (\ell-(b+1))\alpha_2 + \cdots + \alpha_{\ell-b} + \alpha_{n-\ell+a} + 2\alpha_{n - \ell+a+1} + \ldots + (\ell-a)\alpha_{n-1}$, and $k \geq a+b$, $a,b \in \mathbb{Z}_{\geq 1}$, $\max\{a,b\} \leq \ell \leq  \left \lfloor \frac{n+a+b}{2} \right \rfloor-1 $. We show that these weight multiplicities are given by the number of certain pattern avoiding permutations of $\{1, 2, 3, \ldots \ell\}$. 
\end{abstract}

\maketitle
\bigskip
\section{Introduction} 

Affine Lie algebras form an important class of infinite dimensional Kac-Moody Lie algebras. The representation theory of affine Lie algebras have applications in many areas of mathematics and physics. Most of these applications arise from integrable representations of affine Lie algebras. For an affine Lie algebra $\mathfrak{g}$ and a dominant integral weight $\lambda$, there is a unique (up to isomorphism) integrable representation of $\mathfrak{g}$ with highest weight $\lambda$. We denote the corresponding  $\mathfrak{g}$-module  by $V(\lambda)$. Determining the multiplicities of the weights of $V(\lambda)$ is still an open problem. A weight $\mu$ of  $V(\lambda)$ is said to be maximal if $\mu + \delta$ is not a weight where $\delta$ is the null root of $\mathfrak{g}$. Maximal weights form a roof-like structure for the set of weights of $V(\lambda)$. Any maximal weight is Weyl group conjugate to a maximal dominant weight. It is known that there are finitely many maximal dominant weights of $V(\lambda)$ \cite[Proposition 12.6]{Kac}. In order to determine multiplicities of all maximal weights it suffices to determine the multiplicities of the maximal dominant weights. 

%\begin{section}{Preliminaries}\label{prelim}

In this paper we focus on the affine Lie algebra $\mathfrak{g} = \widehat{sl}(n)$ with simple roots $\{\alpha_i \mid 0 \leq i \leq n-1\}$, simple coroots $\{h_i \mid 0 \leq i \leq n-1\}$, fundamental weights $\{\Lambda_i \mid 0 \leq i \leq n-1\}$. Note that $\Lambda_j(h_i) = \delta_{ij}$ and $\alpha_j(h_i) = a_{ij}$ where $A = (a_{ij})$ is the associated affine Cartan matrix where $a_{ii}=2, a_{i,i+1} = -1 = a_{i+1,i}, a_{0,n-1} = -1 =a_{n-1,0}$ and $a_{ij} = 0$ otherwise. The canonical central element and the null root are $c = h_0 + h_1 + \ldots + h_{n-1}$ and $\delta = \alpha_0 + \alpha_1 + \ldots + \alpha_{n-1}$ respectively. The free abelian group $P = \mathbb{Z} \Lambda_0 \oplus  \mathbb{Z} \Lambda_1 \oplus \ldots \oplus  \mathbb{Z} \Lambda_{n-1} \oplus  \mathbb{Z} \delta$ is the weight lattice and $P^+ =  \{ \lambda \in P \mid \lambda({h_i}) \in \mathbb{Z}_{\geq 0} \text{ for all }  i \in I \}$ is the set of dominant integral weights. We consider the integrable highest weight $\widehat{s\ell}(n)$-module $V(k\Lambda_0)$ with highest weight $k\Lambda_0$, $k \in \mathbb{Z}_{\geq 2}$. We denote the set of all maximal weights of $V(k\Lambda_0)$ by $\max(k\Lambda_0)$. The explicit forms of the set of maximal dominant weights $\max(k\Lambda_0) \cap P^+$ are given in \cite[Theorem 3.6]{JM1}. In particular, the weights of the form $k\Lambda_0 - \lambda^\ell_{a,b}$ where $ \lambda^\ell_{a,b} = \ell\alpha_0 + (\ell-b)\alpha_1 + (\ell-(b+1))\alpha_2 + \cdots + \alpha_{\ell-b} + \alpha_{n-\ell+a} + 2\alpha_{n - \ell+a+1} + \ldots + (\ell-a)\alpha_{n-1}$,  $a,b \in \mathbb{Z}_{\geq 1}$, $\max\{a,b\} \leq \ell \leq  \left \lfloor \frac{n+a+b}{2} \right \rfloor-1 $, and $k \geq a+b$, are maximal dominant weights of $V(k\Lambda_0)$. In \cite{Tsu} Tsuchioka showed that for $a=1=b, k=2$ the multiplicities of the maximal dominant weights $2\Lambda_0 - \lambda^\ell_{1,1}$ of $V(2\Lambda_0)$ are given by certain Catalan numbers which are same as the number of $321$-avoiding permutations of $[\ell] = \{1,2, \cdots , \ell\}$. In \cite{JM2}, using the extended Young diagram realization of the crystal $B(k\Lambda_0)$ for the $\widehat{s\ell}(n)$-module $V(k\Lambda_0)$, we proved that the multiplicities of the maximal dominant weights $k\Lambda_0 - \lambda^\ell_{1,1}$ are given by the 
$(k+1)k \cdots 21$ avoiding permutations of $[\ell]$. This result was also obtained in \cite{TW} from a different point of view. More recently, 
Kyu-Hwan Lee and collaborators \cite{KLO} have given weight multiplicities of level $2$ and $3$ maximal dominant weights for other classical affine Lie algebras using some new classes of Young tableaux associated with corresponding crystal bases.

In this paper we generalize our results in \cite{JM2} and determine multiplicities of the maximal dominant weights 
$k\Lambda_0 - \lambda^\ell_{a,b}$ in $V(k\Lambda_0)$. As in \cite{JM2} we use the extended Young diagram realization of the crystal $B(k\Lambda_0)$ for the $\widehat{s\ell}(n)$-module $V(k\Lambda_0)$ which we briefly recall below.

An extended Young diagram $Y = (y_{i})_{i \geq 0}$ of charge $0$ is a weakly increasing sequence with integer entries such that $y_{i} = 0$ for 
$i  \gg  0$.  Associated with each sequence $Y = (y_{i})_{i\geq 0}$ is a unique diagram in the $\mathbb{Z}_{\ge 0} \times \mathbb{Z}$ right half lattice.  For each element $y_{i}$ of the sequence, we draw a column with depth $-y_{i}$, aligned so the top of the column is on the line $y = 0$.  We fill in square boxes for all columns from the depth to the line $y=0$ and obtain a diagram with a finite number of boxes.  We color a box with lower right corner at $(a,b)$ by color $j$, where $(a+b) \equiv j \pmod{n}$.   For simplicity, we refer to color $(n-j)$ by $-j$.   The weight of an extended Young diagram of charge $i$ is $ wt(Y) = \Lambda_{i} - \sum_{j=0}^{n-1}c_{j}\alpha_{j},$ where $c_{j}$ is the number of boxes of color $j$ in the diagram.  For $Y = (y_{i})_{i \geq 0}$ we denote $Y[n] = (y_{i} + n)_{i \geq 0}$. Thus by definition $Y[n]$ is a vertical shift of $Y$ by $n$ units. 

The weight of a $k$-tuple of extended Young diagrams $\Y = (Y_1, Y_2, \ldots , Y_k)$ of charge $0$ is $ wt(\Y) = \sum_{i=1}^{k}wt(Y_{i})$.   Let $\mathcal{Y}(k\Lambda_0)$ denote the set of all $k$-tuples of extended Young diagrams of charge zero.  For two extended Young diagrams $Y=(y_i)_{i \geq 0}$ and $Y'=(y_i')_{i \geq 0}$ we say  $Y \subseteq Y'$ if $y_i \geq y_i'$ for all $i$ which means $Y$ is contained in $Y'$ as a diagram. Note that this containment is transitive.  The realization of the crystal $B(k\Lambda_{0})$ for $V(k\Lambda_0)$ is given in the following theorem.

\begin{theorem}\label{JMMOTh} \cite{JMMO}  Let $V(k\Lambda_0)$ be the irreducible  $\widehat{sl}(n)$-module of highest weight $k\Lambda_0$ and let $B(k\Lambda_0)$ be its crystal.  Then $B(k\Lambda_0) = \{ \Y = (Y_{1}, \ldots , Y_{k}) \in \mathcal{Y}(k\Lambda_0) \mid Y_{1} \supseteq Y_{2} \supseteq \cdots \supseteq Y_{k}  \supseteq Y_1[n], \text{and for each } i \geq 0, \exists \  j \geq 1 \text{ s.t. } (Y_{j+1})_{i} > (Y_{j})_{i+1} \}$.
\end{theorem}

\begin{remark}  Let $B(k\Lambda_0)_{k\Lambda_0-\lambda^\ell_{a,b}}$ denote the set of $\Y \in B(k\Lambda_0)$ such that wt$(\Y) = k\Lambda_0-\lambda^\ell_{a,b}$.  Then mult$_{k\Lambda_0}(k\Lambda_0-\lambda^\ell_{a,b}) = |B(k\Lambda_0)_{k\Lambda_0-\lambda^\ell_{a,b}}|$.
\end{remark}

Our main result is that the multiplicity of $k\Lambda_0 - \lambda^\ell_{a,b}$ in $V(k\Lambda_0)$ equals to the number of $(k+1)k(k-1)\cdots21$-avoiding permutations of $[\ell]$ in which the subsequence of integers 1 through $a$ is in decreasing order and the first $b$ elements are in decreasing order. To prove this result we use the extended Young diagram realization of the crystal $B(k\Lambda_0)$ to express the multiplicity with the number of ordered pairs of certain lattice paths in Section 2. In Section 3, we show that these lattice paths are in one-to-one correspondence with certain standard Young tableaux, hence the multiplicity can be obtained by counting the number of corresponding pairs of tableaux. Finally in Section 4 we use the well-known RSK correspondence and obtain the desired result.

We thank Kyu-Hwan Lee for some discussion in the early stage of this project where he shared with us some data for the multiplicities in the particular case when $a=2, b=1$.
%\end{section}

%%
%%%%%%%%%%%%%%%%%%%%%%%%%%%%%%%%%%%%%%%%
%%%%%%%%%%%%%%%%%%%%%%%%%%%%%%%%%%%%%%%%
%%%%%%%%%%%%%%%%%%%%%%%%%%%%%%%%%%%%%%%%
%%%%%%%%%%%%%%%%%%%%%%%%%%%%%%%%%%%%%%%%
%%%%%%%%%%%%%%%%%%%%%%%%%%%%%%%%%%%%%%%%
%%%%%%%%%%%%%%%%%%%%%%%%%%%%%%%%%%%%%%%%
%%

\begin{section}{Multiplicity and Admissible Sequences of Paths}

\begin{figure}[h]
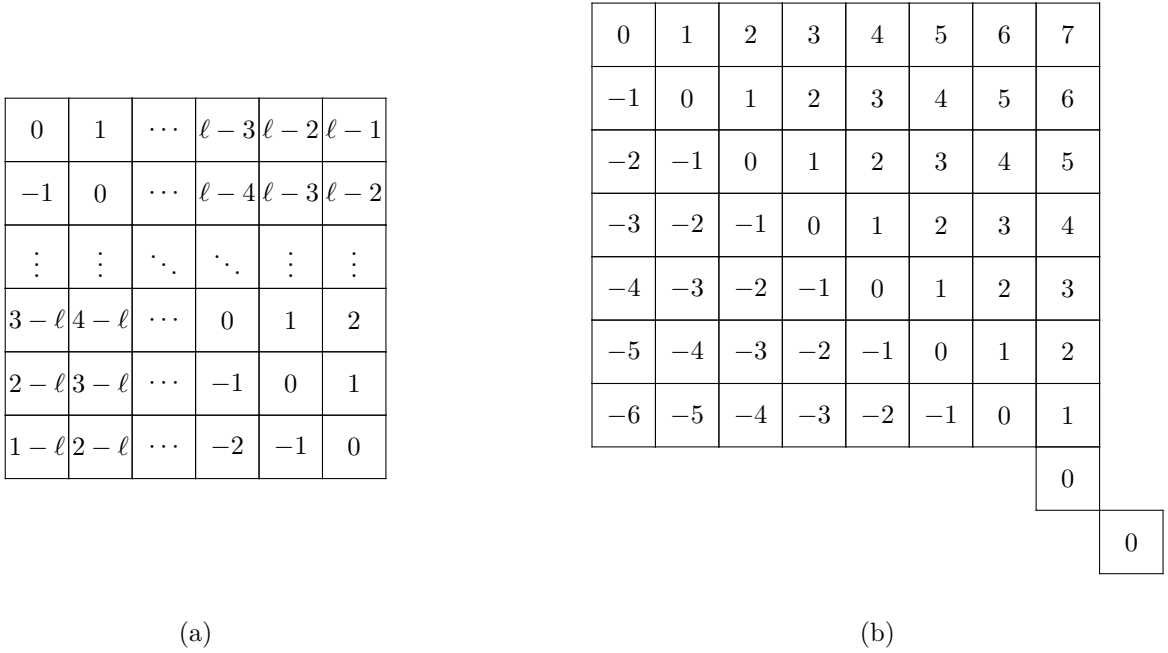


\begin{multicols}{2}
{  $$ \normalsize \tableau{0 & 1 &  \cdots &  \ell - 3& \ell - 2 & \ell -1 \\ -1&  0 & \cdots & \ell - 4 & \ell - 3 & \ell -2\\  \vdots  & \vdots &  \ddots &  \ddots & \vdots & \vdots \\ 3 - \ell & 4-\ell & \cdots & 0 & 1 & 2 \\ 2-\ell  &  3-\ell  & \cdots &-1 &  0 & 1\\ 1-\ell & 2-\ell & \cdots & -2  & -1   & 0 }$$}

(a)

{  $$ \normalsize \tableau{0 & 1 & 2 &   3 & 4 & 5 & 6 & 7 & \fr[l] \\ -1 & 0 & 1 & 2 & 3 & 4 & 5 &  6 &  \fr[l]  \\  -2 & -1 & 0 & 1 & 2 & 3 & 4 & 5 &  \fr[l]   \\ -3 & -2 & -1 & 0 & 1 & 2 & 3 & 4 & \fr[l]  \\ -4 & -3 & -2 & -1 & 0 & 1 & 2 & 3 & \fr[l]  \\  -5 & -4 & -3 & -2 & -1 & 0 & 1 & 2 &\fr[l]   \\ -6 & -5 & -4 & -3 & -2 & -1 & 0 & 1 & \fr[l]  \\  \fr[] &\fr[] &\fr[] &\fr[] &\fr[] &\fr[] &\fr[]   & 0 \\   \fr[] & \fr[] &\fr[] &\fr[] &\fr[] &\fr[] &\fr[] &\fr[]  & 0 \\  }
$$}

(b)

\end{multicols}

\caption{The colored square $Y^\ell$ (a) and $Y^9_{3,2}$ (b)}
\label{thediagramY}
\end{figure}

We consider the colored diagram $Y^\ell$, as in Figure \ref{thediagramY}(a).  This figure is an $\ell \times \ell$ square made up of colored boxes, with color $0$ on the diagonal.  To form the diagram $Y^\ell_{a,b}$ as in Figure \ref{thediagramY}(b), we delete the boxes in the bottom $a-1$ rows to the left of the 0-diagonal  and in the rightmost $b-1$ columns above the 0-diagonal of $Y^\ell$.  We take the upper left corner to be the origin and draw two different types of sequences of $k-1$ lattice paths on the diagram $Y^\ell_{a,b}$ .  The first type of sequence, $\{ p_1^L, p_2^L, \ldots, p_{k-1}^L \}$, is called a sequence of lower paths and begins at $(0,-\ell+a-1)$ and ends $\ell - a$ right ($R$) and up ($U$) moves later on the diagonal of -1 colored boxes (i.e on the line $y = -x - 1$).  The other type of sequence, $\{ p_1^U, p_2^U, \ldots, p_{k-1}^U \}$, is called a sequence of upper paths and begins at $(\ell-b+1,0)$ and ends $\ell-b$ down ($D$) and left ($L$) moves later on the diagonal of 1 colored boxes (i.e on the line $y = -x+1$).  Both types of paths are drawn in such a way that for each color, the number of colored boxes of a specific color below $p_i^L$ (resp. $p_i^U$) is greater than or equal to the number of boxes of that color below $p_{i-1}^L$ (resp. $p_{i-1}^U$).  

For $2 \leq i \leq k-1$, we let $t_i^j$ be the number of boxes of color $j$ between paths  $p_{i-1}^L$ and $p_{i}^L$ for $ j < 0$ and between paths $p_{i-1}^U$ and $p_{i}^U$ for $j > 0$.   We define $t_1^j$ (resp. $t_0^j$) to be the number of boxes of color $j$ below path $p_{1}^L$ (resp. above $p_{k-1}^L$) for $j<0$ and below $p_{1}^U$ (resp. above $p_{k-1}^U$) for $j>0$.  We define admissible sequences of lower and upper paths as follows which is a generalization of the corresponding notions in  \cite[Definition 2.1]{JM2}.

\begin{definition}  \label{pathsdefL} Let $\{ p_1^L, p_2^L, \ldots, p_{k-1}^L \}$  (resp. $\{ p_1^U, p_2^U, \ldots, p_{k-1}^U \}$)  be a sequence of lower  (resp. upper)  paths in $Y^\ell_{a,b}$.  This sequence is an admissible sequence of lower (resp. upper) paths if the following criteria are met.
\begin{enumerate}

\item \label{conds}  $p_1^L$ (resp. $p_1^U$) does not cross the diagonal $y = x - (\ell-a+1)$ (resp. $y = x - (\ell-b+1)$)
\item for  $2 \leq i \leq k-1$, $t_i^j \leq \min \left \{t_{i-1}^j, \ell - |j| -(a-1)- t_1^j - \ds \sum_{m=1}^{i-1}t_m^j \right \}$ for $j \leq -1$\\ (resp.  $t_i^j \leq \min \left \{t_{i-1}^j, \ell - j -(b-1)- t_1^j - \ds \sum_{m=1}^{i-1}t_m^j \right \}$ for $j \geq 1$) and
\item for  $2 \leq i \leq k-1$,  $t_i^j \leq t_i^{j+1}$ for $j < -1$ (resp. $t_i^j \leq t_i^{j-1}$ for $j > 1$)

\end{enumerate}
\end{definition}

We associate with each admissible sequence of lower paths a partition $\tau = (t_0^{-1}, t_1^{-1}, t_2^{-1}, \ldots, t_{k-1}^{-1}) \vdash \ell -a$ and say that the sequence is of type $\tau$.   With each admissible sequence of upper paths we associate a partition $\omega = (t_0^{1}, t_1^1, t_2^1, \ldots, t_{k-1}^1) \vdash \ell-b$, and say that the sequence is of type $\omega$.  For a particular $\tau$ (resp. $\omega$), we let $L^{\ell,k,\tau}_a$ (resp. $U^{\ell,k,\omega}_b$) denote the set of all admissible lower (resp. upper) sequences of type $\tau$ (resp. $\omega$).

Let $\mu$ be a partition of some integer $\ell$ with at least $j$ rows and let $P^{\mu}_j$ be the set of partitions of $\ell-j$ such that exactly $j$ distinct parts of $\mu$ are decreased by 1.  For example,  $P^{(3,2,2,1)}_2 = \{(2,2,1,1), (2,2,2), (3,1,1,1),  (3,2,1)\}$.  For a partition $\mu$, we let $l(\mu)$ be the length of $\mu$, which is the number of parts. Let $\mu \vdash \ell$, with $\max\{a,b\} \leq l(\mu) \leq k$ and define $\mc{S}^{\ell,k,\mu}_{a,b}$ to be the set of all ordered pairs $({\bf L},{\bf U})$, where ${\bf L} \in L^{\ell,k,\tau}_a$ for any $\tau \in P_a^\mu$ and ${\bf U} \in U^{\ell,k,\omega}_b$ for any $\omega \in P_b^\mu$.

\begin{example} \label{LUex}  
Let $\mu = (3,2,2,1,1)$, $\ell = 9$, $a = 3, b=2$, $k=5$. In Figure \ref{LUfig}(a), we have an admissible sequence of lower paths $({\bf L} = \{p_1^L, p_2^L, p_3^L, p_4^L\})$ and a table with the moves of each path.  Note that $t_0^{-1} = 2, t_1^{-1} = 1, t_2^{-1} = 1, t_3^{-1} = 1, t_4^{-1} = 1$ and so $\mathbf{L}$ is of type $\tau=  (2,1,1,1,1) \in P_3^{(3,2,2,1,1)}$.  In Figure \ref{LUfig}(b), we have an admissible sequence of upper paths $({\bf U} = \{p_1^U, p_2^U, p_3^U, p_4^U\})$ and the associated table.  We see that $t_0^{1} = 2, t_1^{1} = 2, t_2^{1} = 1, t_3^{1} = 1, t_4^{1} = 1$ and so $\mathbf{U}$ is of type $\omega = (2,2,1,1,1)\in P_2^{(3,2,2,1,1)}$.  So $({\bf L},{\bf U}) \in \mc{S}^{9,5,(3,2,2,1,1)}_{3,2}$.

\begin{figure}[h]

\begin{multicols}{2}

\includegraphics[scale=.6]{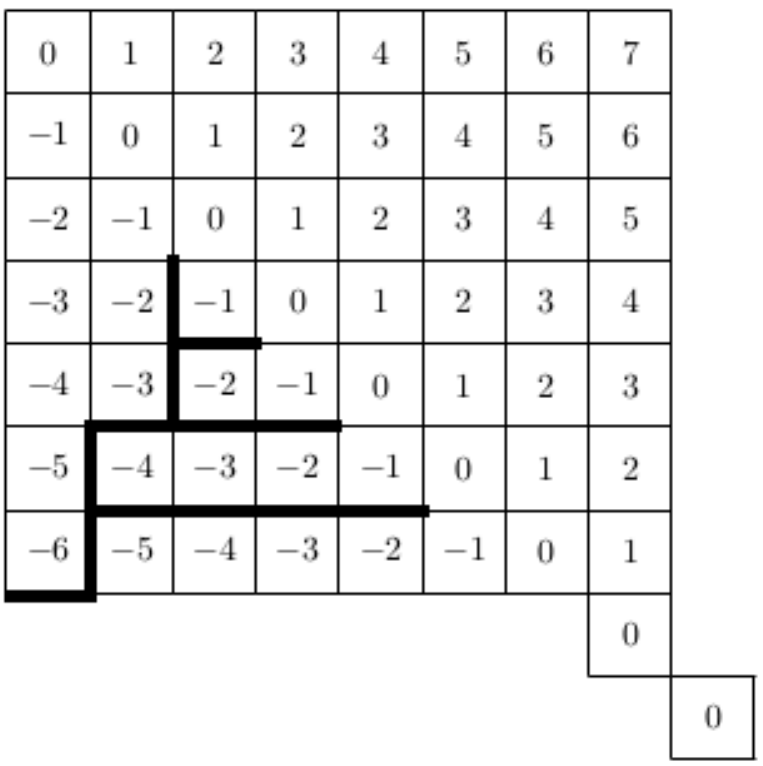}

\scalebox{.8}{
\begin{tabular}{lllllll} 
\hline\noalign{\smallskip}
 & \multicolumn{6}{c} {Move} \\
 \noalign{\smallskip}\hline\noalign{\smallskip}
Path & 1 & 2 & 3 & 4 & 5 & 6   \\ 
\noalign{\smallskip}\hline\noalign{\smallskip}
1 & $R$ & $U$ & $R$ & $R$ & $R$ & $R$ \\
2 & $R$ & $U$ & $U$ & $R$ & $R$ & $R$  \\ 
3 & $R$ & $U$ & $U$ & $R$ & $U$ & $R$  \\ 
4 & $R$ & $U$ & $U$ & $R$ & $U$ & $U$ \\ 
\noalign{\smallskip}\hline
\end{tabular} }

\includegraphics[scale=.6]{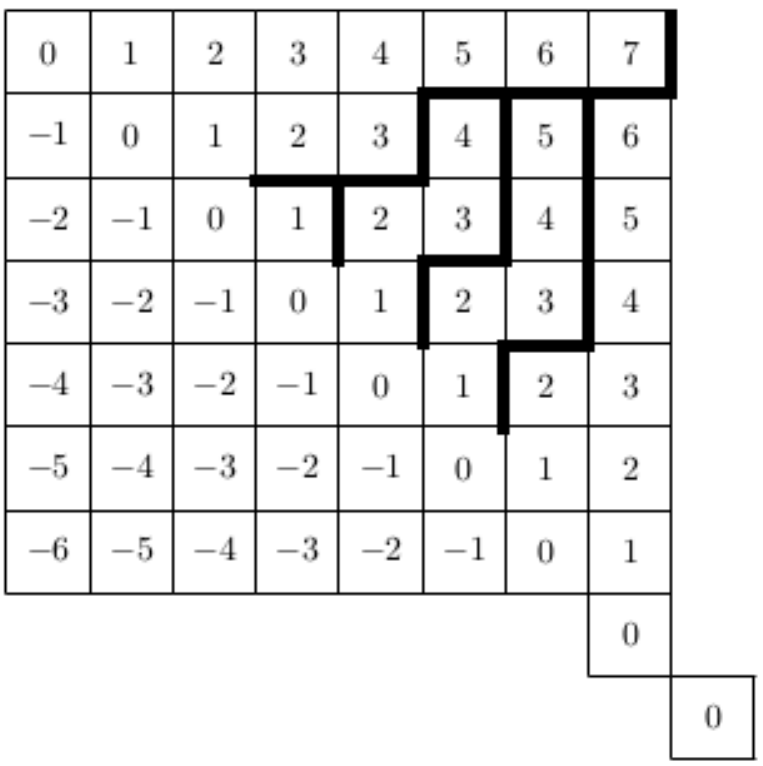}

\scalebox{.8}{
\begin{tabular}{llllllll} 
\hline\noalign{\smallskip}
 & \multicolumn{7}{c} {Move} \\
 \noalign{\smallskip}\hline\noalign{\smallskip}
Path & 1 & 2 & 3 & 4 & 5 & 6  & 7  \\ 
\noalign{\smallskip}\hline\noalign{\smallskip}
1 & $D$ & $L$ & $D$ & $D$ & $D$ & $L$ & $D$ \\ 
2 & $D$ & $L$ & $L$ & $D$ & $D$ & $L$ & $D$ \\ 
3 & $D$ & $L$ & $L$ & $L$ & $D$ & $L$ & $D$ \\ 
4 & $D$ & $L$ & $L$ & $L$ & $D$ & $L$ & $L$ \\ 
\noalign{\smallskip}\hline
\end{tabular} }

\end{multicols}

\begin{multicols}{2}
(a)

(b)
\end{multicols}

\caption{Example of admissible  sequences of lower/upper paths}
\label{LUfig}
\end{figure}
\end{example}

We remark that an admissible sequence of lower paths in $Y^\ell_{a,b}$ consists of the first $\ell-a$ moves of each path in an admissible sequence of $k-1$ paths  in the $(\ell-a+1) \times (\ell-a+1)$ square (see \cite[Definition 4.5]{JM1} and \cite[Definition 2.1]{JM2}).
 
\begin{theorem}\label{mult_admiss} Let $a, b \geq 1$.  For  $k \geq a+b$,  $\max\{a,b\} \leq \ell \leq  \left \lfloor \frac{n+a+b}{2} \right \rfloor -1 $, 
 $$\text{mult}_{V(k\Lambda_0)}(k\Lambda_0 - \lambda_{a,b}^\ell) =  \ds \left |\bigcup_{\mu \vdash \ell, \max\{a,b\} \leq l(\mu) \leq k} \mc{S}^{\ell,k,\mu}_{a,b} \right |.$$  
\end{theorem}

\begin{proof}  It is enough to show that the elements in $\bigcup_{\mu \vdash \ell, \max\{a,b\} \leq l(\mu) \leq k} \mc{S}^{\ell,k,\mu}_{a,b}$ are in one-to-one correspondence with the $k$-tuples of extended Young diagrams in $B(k\Lambda_0)_{k\Lambda_0-\lambda^\ell_{a,b}}$.

 Let $\mu  = (\mu_1, \mu_2, \ldots, \mu_k) \vdash \ell, \max\{a,b\} \leq l(\mu) \leq k$, and let $({\bf L},{\bf U})\in \mc{S}^{\ell,k,\mu}_{a,b}$ be comprised of an admissible sequence of lower paths ${\bf L} =\{ p_1^L, p_2^L, \ldots, p_{k-1}^L \}$ of type $\tau \in P^\mu_a$  and an admissible sequence of upper paths ${\bf U} =  \{ p_1^U, p_2^U, \ldots, p_{k-1}^U \}$ of type $\omega \in P^\mu_b$. We construct the $k$-tuple ${\bf Y}$ of extended Young diagrams as follows. Remove the boxes of colors $j < 0$ below $p_1^L$, $\mu_2$ boxes of color 0,  and the boxes of colors $j > 0$ below $p_1^U$.  Since $\tau \in P^\mu_a$ and $\omega \in P^\mu_b$  (and so $\mu_2 = \tau_2$ or $\tau_2+ 1$ and $\mu_2 = \omega_2$ or $\omega_2+1$) and $p_1^L$ and $p_1^U$ are lattice paths, we can uniquely form an extended Young diagram, $Y_2$, from these removed boxes.  Next, we consider the boxes between $p_1^L$ and $p_2^L$,  $\mu_{3}$ 0-colored boxes, and the boxes  between $p_1^U$ and $p_2^U$ and use them to form $Y_3$.    We continue this process for the boxes between subsequent paths, until the boxes between $p_{k-1}^L$ and $p_{k-2}^L$, $\mu_{k}$ boxes of color 0, and the boxes between $p_{k-1}^U$ and $p_{k-2}^U$ have been used to form $Y_k$.  Finally, we take the boxes remaining above $p_{k-1}^L$ and $p_{k-1}^U$, together with $\mu_1$ boxes of color 0, to be $Y_1$.   As in \cite[proof of Theorem 4.9]{JM1} and because $\tau \in P_a^\mu, \omega \in P_b^\mu,$ each $Y_i$ is an extended Young diagram and we have $Y_1 \supseteq Y_2 \supseteq \cdots \supseteq Y_k$.  Note that all diagrams in the tuple $\mathbf{Y} = (Y_1, Y_2, \ldots, Y_k)$ collectively consist of exactly the boxes in $Y^\ell_{a,b}$ and so  ${\bf Y}$ has weight $k\Lambda_0 - \lambda^\ell_{a,b}$.  Finally, consider $Y_k$  as a sequence represented by $Y_k = (y_i^{(k)})_{i \geq 0}$.  We write $Y_{k+1} = Y_1[n] = (y_i^{(1)}+n)_{i \geq 0}$.  Since $n \geq 2\ell+2-a-b$ and $y_i^{(1)} \geq -\ell+a-1$ for all $i \geq 0$, we have $y_i^{(1)}+n \geq \ell - b + 1 > 0$ for all $i \geq 0$.  Note that by definition $(Y_k)_i = y_i^{(k)} \leq 0$ for all $i \geq 0$.   Hence $Y_k \supseteq Y_1[n]$ and  for all $i \geq 0$, $(Y_{k+1})_i > (Y_k)_{i+1}$.  Therefore, ${\bf Y} \in B(k\Lambda_0)_{k\Lambda_0-\lambda^\ell_{a,b}}$.
 
Now, let ${\bf Y} = (Y_1, Y_2, \ldots, Y_k)  \in B(k\Lambda_0)_{k\Lambda_0 - \lambda^\ell_{a,b}}$.   Let $\mu = (\mu_1, \mu_2, \ldots, \mu_k) \vdash \ell$, where $\mu_i$ is the number of 0-colored boxes in $Y_i$.   Let $\tau = (\tau_1, \tau_2, \ldots, \tau_k) \vdash \ell-a$ (resp. $\omega = (\omega_1, \omega_2, \ldots, \omega_k) \vdash \ell-b$), where $\tau_i$ (resp. $\omega_i$) is the number of -1-colored (resp. 1-colored) boxes in $Y_i$.  Since each $Y_i$ is an extended Young diagram, we see that $\tau \in P^\mu_a$ (resp. $\omega \in P^\mu_b$).  The weight of ${\bf Y}$ is $k\Lambda_0 - \lambda^\ell_{a,b}$ and so we have the appropriate number of boxes in the appropriate colors to fill the diagram $Y^\ell_{a,b}$ with boxes from ${\bf Y}$.  To do so, we first place $Y_1$ in $Y^\ell_{a,b}$, using gravity to the upper left corner.   Starting at the bottom left and continuing until we reach the -1-diagonal, we draw the path $p_{k-1}^L$ along the outside edge.  Then, starting at the upper right corner of $Y_1$, we draw the path $p_{k-1}^U$ down and left along the outside edge until we reach the 1-diagonal.   Because $Y_1$ is a Young diagram, both of these paths are lattice paths.  Now, we take the boxes in $Y_k$ and place them in $Y^\ell_{a,b}$, again with gravity to the upper left corner.  We draw paths $p_{k-2}^L$ and $p_{k-2}^U$  along the outside edge from the bottom left to the -1-diagonal and from the upper right to the 1-diagonal, respectively.  Because $Y_1$ and $Y_k$ are extended Young diagrams, these paths must be lattice paths.  Next, we take the boxes in $Y_{k-2}$ and draw $p_{k-3}^L$ and $p_{k-3}^U$.  We continue in this manner until we add in the final boxes of $Y_2$ and insert the $\ell$ boxes of color 0 to completely fill in the diagram $Y^\ell_{a,b}$.  Let ${\bf L}$ and ${\bf U}$ be the sequences of lower and upper paths.    As in  \cite[proof of Theorem 4.9]{JM1}, since each $Y_i$ is an extended Young diagram, Definition \ref{pathsdefL}(3) is satisfied for ${\bf L}$, $\mathbf{U}$. In addition, since $Y_1 \supseteq Y_2 \supseteq \cdots \supseteq Y_k$, we see that Definition \ref{pathsdefL}(1) and Definition \ref{pathsdefL}(2) are satisfied.  Thus ${\bf L} \in L^{\ell,k,\tau}_a$ for $\tau \in P_a^\mu$,  ${\bf U} \in U^{\ell,k,\omega}_b$ for $\omega \in P_b^\mu$, and $({\bf L},{\bf U}) \in  \mc{S}^{\ell,k,\mu}_{a,b}$. \end{proof}

\begin{example}
We associate the element $(\mathbf{L}, \mathbf{U})$ of $\mc{S}^{9,5,(3,2,2,1,1)}_{3,2}$ in Example \ref{LUex} with the 5-tuple of extended Young diagrams in Figure \ref{EYDtuple}.  

\begin{figure}[h]
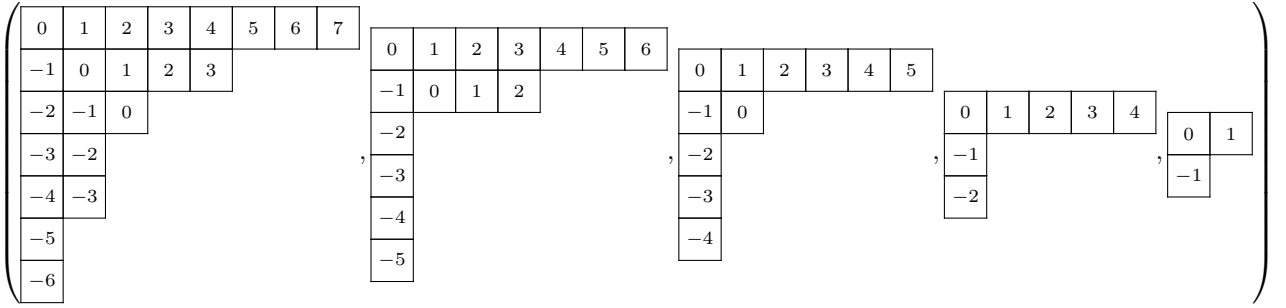

$$\left ( \scriptsize{\tableau{0 & 1 & 2 & 3 & 4 & 5 & 6 & 7 \\ -1 & 0 & 1 & 2 & 3 \\ -2 & -1 &0 \\ -3 & -2   \\ -4 & -3 \\ -5 \\ -6 } , 
\tableau{0 & 1 & 2 & 3 & 4 & 5& 6 \\ -1 & 0 & 1  & 2 \\ -2 \\ -3 \\ -4 \\ -5} , 
\tableau{0 & 1 & 2 & 3  & 4 & 5  \\ -1  &0\\ -2  \\ -3 \\ -4} , 
\tableau{0 & 1 & 2  & 3 & 4 \\ -1 \\ -2    }, 
\tableau{ 0 & 1\\ -1 } }\right )
$$
\caption{Element of $B(5\Lambda_0)_{5\Lambda_0-\lambda^9_{3,2}}$}
\label{EYDtuple}
\end{figure}
\end{example}
\end{section}

%%
%%%%%%%%%%%%%%%%%%%%%%%%%%%%%%%%%%%%%%%%
%%%%%%%%%%%%%%%%%%%%%%%%%%%%%%%%%%%%%%%%
%%%%%%%%%%%%%%%%%%%%%%%%%%%%%%%%%%%%%%%%
%%%%%%%%%%%%%%%%%%%%%%%%%%%%%%%%%%%%%%%%
%%%%%%%%%%%%%%%%%%%%%%%%%%%%%%%%%%%%%%%%
%%%%%%%%%%%%%%%%%%%%%%%%%%%%%%%%%%%%%%%%
%%

\begin{section}{Admissible Sequences of Paths and Standard Young Tableaux}
In this section, we relate the sequences of paths to standard Young tableaux, leading to a formula for the multiplicity of the weights $k\Lambda_0 - \lambda^\ell_{a,b}$.  For a partition $\mu$, we let $T^\mu$ be the set of all standard Young tableaux of shape $\mu$.  We denote $l(\mu)$ to be the number of parts in $\mu$.  

For $\max\{a,b\} \leq \ell \leq  \left \lfloor \frac{n+a+b}{2} \right \rfloor-1$ let $\tau$ be a partition of $\ell-a$ with $\max\{a,b\}  \leq l(\tau) \leq k$.  Similar to \cite[Section 2]{JM2}, we define a map $f \colon T^\tau \to L^{\ell,k,\tau}_a$ as follows.  Let $X \in T^\tau$ and consider the diagram $Y^\ell_{a,b}$.  We draw a sequence of lower paths $f(X)$ with right ($R$) and up ($U$) moves as follows.  For $1 \leq i \leq k-1$, $1 \leq j \leq \ell-a$,  let $m_{i,j}$ denote the $j^{th}$ move in $p_i^L$.  Then we define 
$$m_{i,j} = \begin{cases}
U, & \text{ if $j$ is in row 2 through $(i+1)$ of $X$} \\
R, & \text{ otherwise} \\
\end{cases}.
$$ 
Next, we define a map $g \colon L^{\ell,k,\tau}_a\to T^\tau$ as follows.  For ${\bf L} = \{ p_1^L, p_2^L, \ldots, p_{k-1}^L \} \in L^{\ell,k,\tau}_a$, define $g( {\bf L})$ by first drawing a Young diagram $X$ with shape $\tau= (t_0^{-1}, t_1^{-1}, \ldots, t_{k-1}^{-1})$.  Now we fill in the boxes of $X$ with numbers $1, 2, \ldots , \ell-a$ by first traversing the first $\ell-a$ moves from $(\ell-a+1,0)$ to the $-1$-diagonal in each path $p_1^L, p_2^L, \ldots, p_{k-1}^L$ in order.  If the $j^{th}$ move $m_{i,j} = R$ for all $1 \leq i \leq k-1$, then we place the number $j$ in the leftmost available box in the first row of $X$.  If $m_{i,j} = U$ for any $1 \leq i \leq k-1$, then choose $s$ to be the smallest integer such that $m_{s,j} = U$. In such case, we place the number $j$ in the leftmost available box in row $s+1$ of $X$.

Replacing $\ell$ by $\ell-a$ in the definition of admissible sequence of paths and using the same argument in the proof of  \cite[Theorem 2.3]{JM2}, $f(X) \in L_a^{\ell,k,\tau}$, $g(\mathbf{L}) \in T^\tau$ and these functions are inverses of each other.  Thus, the following lemma holds.   

\begin{lemma}\label{bijL} For $\max\{a,b\} \leq \ell \leq  \left \lfloor \frac{n+a+b}{2} \right \rfloor-1$ let $\tau$ be a partition of $\ell-a$ with $\max\{a,b\}  \leq l(\tau) \leq k$.  Then $|T^\tau| = |L^{\ell,k,\tau}_a|$.  
\end{lemma}

By a similar argument, replacing right moves with down moves and up moves with left moves, the following lemma holds for upper paths.  

\begin{lemma}\label{bijU} For $\max\{a,b\} \leq \ell \leq  \left \lfloor \frac{n+a+b}{2} \right \rfloor-1$ let $\omega$ be a partition of $\ell-b$ with $\max\{a,b\}  \leq l(\omega) \leq k$.  Then $|T^\omega| = |U^{\ell,k,\omega}_b|$.  
\end{lemma}

\begin{example}  Let $k = 4$, $\ell = 10$, $a = 3$, $b=1$.  
 We begin with the standard Young tableau in Figure \ref{exfL}(a) and applying $f$, we list the moves in $p_1^L, p_2^L, p_3^L$ in Figure \ref{exfL}(b) and draw them in Figure \ref{exfL}(c).  
\begin{figure}[h]
\begin{centering}
\begin{multicols}{3}

$\scalebox{.9}{\tableau{1 & 2 & 6 \\ 3 & 5  \\ 4  \\ 7}}$

\scalebox{.8}{
\begin{tabular}{llllllll} 
\hline\noalign{\smallskip}
 & \multicolumn{7}{c} {Move} \\
 \noalign{\smallskip}\hline\noalign{\smallskip}
Path & 1 & 2 & 3 & 4 & 5 & 6 & 7  \\ 
\noalign{\smallskip}\hline\noalign{\smallskip}
1 & $R$ & $R$ & $U$ & $R$ & $U$ & $R$ & $R$ \\
2 & $R$ & $R$ & $U$ & $U$ & $U$ & $R$ & $R$ \\ 
3 & $R$ & $R$ & $U$ & $U$ & $U$ & $R$ & $U$ \\ 
\noalign{\smallskip}\hline
\end{tabular} }

\includegraphics[scale=.55]{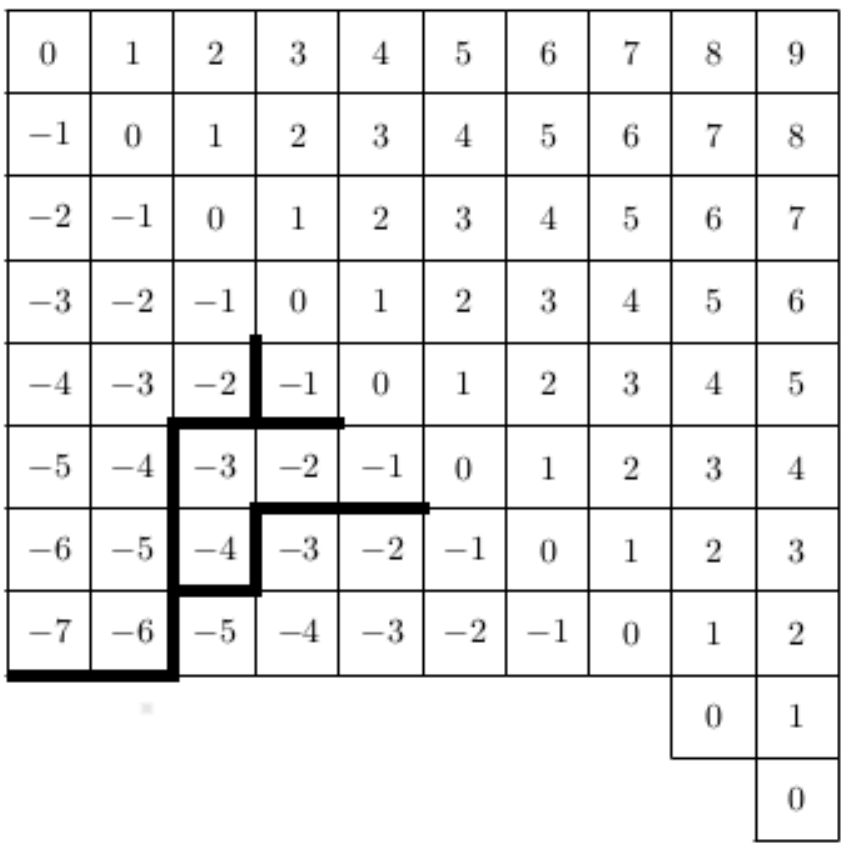}
\end{multicols}

\begin{multicols}{3}
(a)

(b)

(c)

\end{multicols}
\end{centering}
\caption{Example of the map $f$}
\label{exfL}
\end{figure}
\end{example}

\begin{example}\label{running_g} 
Applying $g$ to $\mathbf{L}$ in Example \ref{LUex}(a), we see that $\mathbf{L}$ corresponds to the standard Young tableau in Figure \ref{exg}(a). Similarly, $\mathbf{U}$ from Example \ref{LUex}(b) corresponds with the standard Young tableau in Figure \ref{exg}(b).  

\begin{figure}[h]
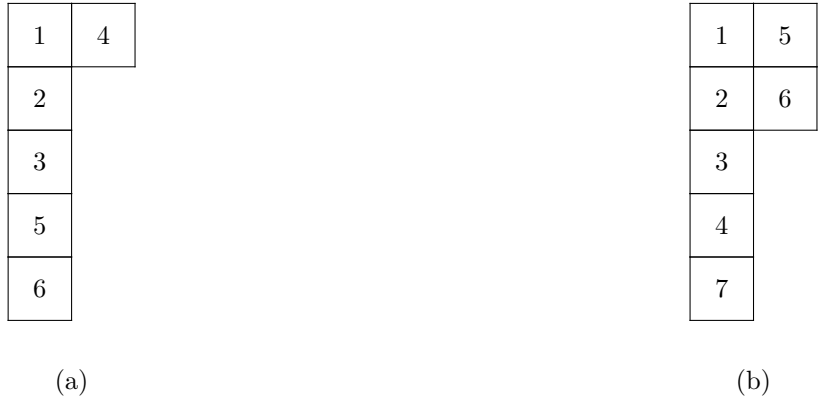

\begin{centering}
\begin{multicols}{2}

$ \tableau{1 & 4 \\ 2  \\ 3\\ 5 \\ 6}$

$\tableau{1 & 5 \\ 2 & 6 \\ 3 \\ 4\\7}$

\end{multicols}

\begin{multicols}{2}
(a)

(b)

\end{multicols}
\end{centering}

\caption{Example of the map $g$}
\label{exg}
\end{figure}
\end{example}

In the following theorems, we note that the multiplicity of $k\Lambda_0-\lambda_{a,b}^\ell$ can be given by the number of ordered pairs of certain standard Young tableaux and hence we obtain a formula for its multiplicity.  

\begin{theorem} \label{formulaLemma} Let $k \geq a+b$, $\max\{a,b\} \leq \ell \leq  \left \lfloor \frac{n+a+b}{2} \right \rfloor-1 $. Then we have 
%$$\text{mult}_{V(k\Lambda_0)}(k\Lambda_0 - \lambda_{a,b}^\ell) = \left |\bigcup_{\scriptsize{\stackunder{$\mu\vdash\ell, \max\{a,b\} \leq  l(\mu) \leq k,$}{$\nu \in T^\tau (\tau \in P^\mu_a), \rho \in T^\omega (\omega \in P^\mu_b)$}}} (\nu, \rho)  \right | .$$
$$\text{mult}_{V(k\Lambda_0)}(k\Lambda_0 - \lambda_{a,b}^\ell) = \left |\bigcup_{\mu\vdash\ell, \max\{a,b\} \leq  l(\mu) \leq k} \left ( \bigcup_{\scriptsize{\stackunder{$V \in T^\tau (\tau \in P^\mu_a),$}{$W \in T^\omega (\omega \in P^\mu_b)$}}} (V, W)  \right ) \right | .$$
\end{theorem}

\begin{proof}
For $\mu \vdash \ell, \max\{a,b\} \leq l(\mu) \leq k$, recall that $\mc{S}^{\ell,k,\mu}_{a,b}$ is the set of  all $({\bf L},{\bf U})$, where ${\bf L} \in L^{\ell,k,\tau}_a$ for some $\tau \in P^\mu_a$ and ${\bf U} \in U^{\ell,k,\omega}_b$ for $\omega \in P^\mu_b$. By Lemma \ref{bijL} and Lemma \ref{bijU}, each such ordered pair $({\bf L},{\bf U})$ is in one-to-one correspondence with an ordered pair $(V, W)$ of standard Young tableau where $V \in T^\tau$ for some $\tau \in P^\mu_a$ and $W \in T^\omega$ for some $\omega \in P^\mu_b$.  By Theorem \ref{mult_admiss}, the number of such ordered pairs gives the multiplicity of $k\Lambda_0 - \lambda^\ell_{a,b}$.  \end{proof}

For a partition $\mu$, we let $f^\mu$ denote the number of standard Young tableaux of shape $\mu$ and recall that this can be computed using the Frame-Robinson-Thrall hook-length formula \cite{FRT}. Note that $|T^\mu| = f^\mu$.  Hence the following result is an immediate consequence.

\begin{theorem} \label{formulaTh} For $k \geq a+b$, $\max\{a,b\} \leq \ell \leq  \left \lfloor \frac{n+a+b}{2} \right \rfloor-1 $, we have 
$$\text{mult}_{V(k\Lambda_0)}(k\Lambda_0 - \lambda_{a,b}^\ell) = \displaystyle \sum_{\mu \vdash \ell, \max\{a,b\} \leq l(\mu) \leq k} \left (\sum_{\tau \in P^\mu_a} f^{\tau} \right) \cdot \left (\sum_{\omega \in P^\mu_b} f^{\omega} \right). $$
\end{theorem}

\begin{proof} For a given partition $\mu \vdash \ell$ in Lemma \ref{formulaLemma}, there are  $\ds \sum_{\tau \in P^\mu_a} f^{\tau}$ choices for $V$ and  $\ds \sum_{\omega \in P^\mu_b} f^{\omega}$choices for $W$.  \end{proof}

\begin{example}
Consider the weight $5\Lambda_0-\lambda_{3,1}^6$ in $V(5\Lambda_0)$.  There are  six partitions of $\ell = 6$ with between 3 and 5 parts:   $(2,2,2),(3,2,1),(4,1,1), (2,2,1,1), (3,1,1,1), (2,1,1,1,1)$.  Consider, for example,  $(2,2,1,1)$.  Then $P_3^{(2,2,1,1)} = \{(2,1), (1,1,1)\}$ and $P_1^{(2,2,1,1)} = \{(2,1,1,1), (2,2,1)\}$.  Using the hook-length formula, we calculate $f^{(2,1)} =2 , f^{(1,1,1)} = 1, f^{(2,1,1,1)} = 4 , f^{(2,2,1)}=5$ and so $ \left (\sum_{\tau \in P^{(2,2,1,1)}_3} f^{\tau} \right) \cdot \left (\sum_{\omega \in P^{(2,2,1,1)}_1} f^{\omega} \right) = 3\cdot9=27$.  We do the same for the other partitions of $\ell=6$ and obtain:
$$\text{mult}_{V(5\Lambda_0)}(5\Lambda_0 - \lambda_{3,1}^6) = (1)(5) + (2)(5+6+5) + (1)(6+4) + (2+1)(4+5) + (1+2)(4+6) + (1+2)(1+4)=119$$
\end{example}

\end{section}

%%%%%%%%%%%%%%%%%%%%%%%%%%%%%%%%%%%%%%%%
%%%%%%%%%%%%%%%%%%%%%%%%%%%%%%%%%%%%%%%%
%%%%%%%%%%%%%%%%%%%%%%%%%%%%%%%%%%%%%%%%
%%%%%%%%%%%%%%%%%%%%%%%%%%%%%%%%%%%%%%%%
%%%%%%%%%%%%%%%%%%%%%%%%%%%%%%%%%%%%%%%%
%%%%%%%%%%%%%%%%%%%%%%%%%%%%%%%%%%%%%%%%

\begin{section}{Multiplicity and Avoiding Permutations}

In this section, we will show that the multiplicity of the weights $k\Lambda_0 - \lambda^\ell_{a,b}$ can be counted by certain avoiding permutations.  

First, for completeness we recall a sliding operation developed by Sch{\" u}tzenberger.  This procedure is often called the jeu de taquin, after a game with sliding pieces (c.f. \cite{F}).    We begin with a skew Young tableau, say $\mu \setminus \lambda$.  An inside corner is a box in the deleted diagram ($\lambda$) such that the boxes immediately to its right and below it are not in the deleted diagram.   An outside corner is a box in the portion of the diagram that has not been deleted that does not have a box below it or to its right.  The procedure is to take an inside corner (called a hole) and look at the neighbors immediately to the right of the hole and immediately below the hole.  We slide whichever is smaller into the hole, thus moving the hole.  If there is only one right or below neighbor, we slide that box into the hole. For the new position of the hole, we again look at neighbors to the right and below and slide accordingly until we are at an outside corner, at which time we consider the hole to be removed.  Once this sliding procedure is performed for each inside corner, the result is known to be a tableau, called the rectification of the skew diagram.   The rectification of a skew diagram into a  tableau is known to be reversible (c.f. \cite{F}), given that we know which boxes were removed and from which outside corner.  

\begin{example}  We show the rectification of ${\scriptsize \tableau{1 & 3 & 6 & 9 \\ 2 & 4 & 8 \\ 5 & 7 & 10}}  \setminus  {\scriptsize \tableau{1 \\ 2}}$ in Figure \ref{rectex}.  After the rectification is complete, we subtract two from each entry to obtain a standard Young tableaux with 10-2 = 8 boxes (shown in last step).  Observe that the box that contained a 2 in the deleted diagram is removed in row 3 and the box that contained a 1 in the deleted diagram is removed from row 2. 

\begin{figure}[h]
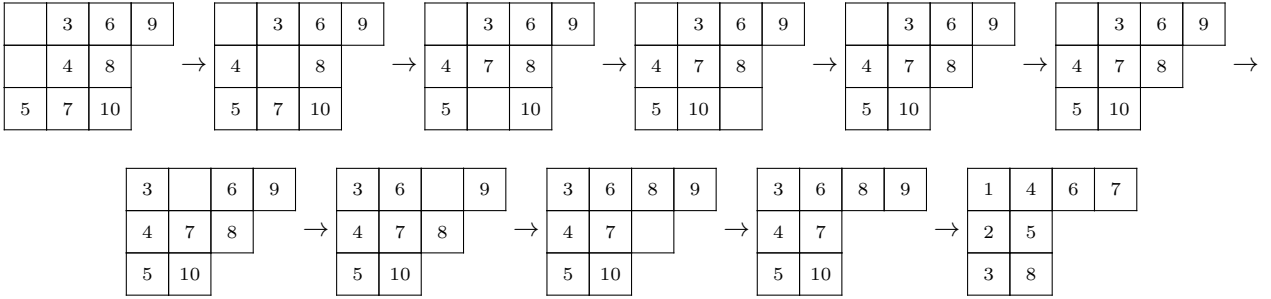

$$ {\scriptsize \tableau{ & 3 & 6 & 9 \\  & 4 & 8 \\ 5 & 7 & 10}} \to 
{\scriptsize \tableau{ & 3 & 6 & 9 \\  4&  & 8 \\ 5 & 7 & 10}} \to
{\scriptsize \tableau{ & 3 & 6 & 9 \\  4&7  & 8 \\ 5 &   & 10}} \to 
{\scriptsize \tableau{ & 3 & 6 & 9 \\  4&7  & 8 \\ 5 &   10 & }} \to 
{\scriptsize \tableau{ & 3 & 6 & 9 \\  4&7  & 8 \\ 5 &   10  }} \to 
{\scriptsize \tableau{ & 3 & 6 & 9 \\  4&7  & 8 \\ 5 &   10  }} \to $$

$${\scriptsize \tableau{3 &  & 6 & 9 \\  4&7  & 8 \\ 5 &   10  }} \to 
{\scriptsize \tableau{3 & 6 &  & 9 \\  4&7  & 8 \\ 5 &   10  }} \to
{\scriptsize \tableau{3 & 6 & 8 & 9 \\  4&7  &  \\ 5 &   10  }} \to
{\scriptsize \tableau{3 & 6 & 8 & 9 \\  4&7    \\ 5 &   10  }} \to
{\scriptsize \tableau{1 & 4 & 6 & 7 \\  2&5    \\ 3 &   8  }}
$$

\caption{Example of rectification.}
\label{rectex}
\end{figure}

\end{example}

Let $\mu \vdash \ell$, with $l(\mu) > a$ and define $Q_a^\mu$ be the set of standard Young tableau of shape $\mu$ such that the first column includes entries 1 through $a$.  For $M \in Q_a^\mu$, remove the boxes with entries 1 through $a$, and perform the process of rectification.  Because all the boxes in the deleted diagram are in the first column, we will remove the holes from bottom to top. The process ensures that each hole is removed from a row higher than the previous box was removed.  We subtract $a$ from each box in the rectification and obtain a standard Young tableau in $T^\tau$ for some $\tau \in P_a^\mu$.  Since we can compare $\tau$ and $\mu$ to see which rows had boxes removed, the process is reversible.  Hence using Theorem \ref{formulaLemma}, we obtain the following result.

\begin{theorem}\label{pairsforRSKLemma}
Let $k \geq a+b$, $\max\{a,b\} \leq \ell \leq  \left \lfloor \frac{n+a+b}{2} \right \rfloor-1 $. Then we have 
$$\text{mult}_{V(k\Lambda_0)}(k\Lambda_0 - \lambda_{a,b}^\ell) = \left |\bigcup_{\mu\vdash\ell, \max\{a,b\} \leq  l(\mu) \leq k} \left (  \bigcup_{ M \in Q_a^\mu, N \in Q_b^\mu} (M,N)  \right ) \right | .$$
\end{theorem}

\begin{example}\label{running_rectification} Let $a=3, b=2, \ell=9, k=5, \mu = (3,2,2,1,1)$.  In Figure \ref{exrect}(a), we begin with an element of  $Q_3^{(3,2,2,1,1)}$, remove the tableau with entries 1, 2, and 3,  and finally give the rectification of the diagram.  Each box is removed from a different row and the resulting diagram has shape $\tau = (2,1,1,1,1) \in P_3^{(3,2,2,1,1)}$.    Notice that when we subtract $a=3$ from each entry, we obtain the standard Young tableau in Figure \ref{exg}(a) of Example  \ref{running_g}.   Similarly, in Figure \ref{exrect}(b), we take an element in $Q_2^{(3,2,2,1,1)}$, remove the tableau with entries 1 and 2, and give the rectification of the diagram. Each box is removed from a different row and the resulting diagram has shape $\omega =  (2,2,1,1,1) \in P_2^{(3,2,2,1,1)}$.  When we subtract $b=2$ from each entry, we obtain the standard Young tableau in Figure \ref{exg}(b) of Example  \ref{running_g}.

\begin{figure}[h]
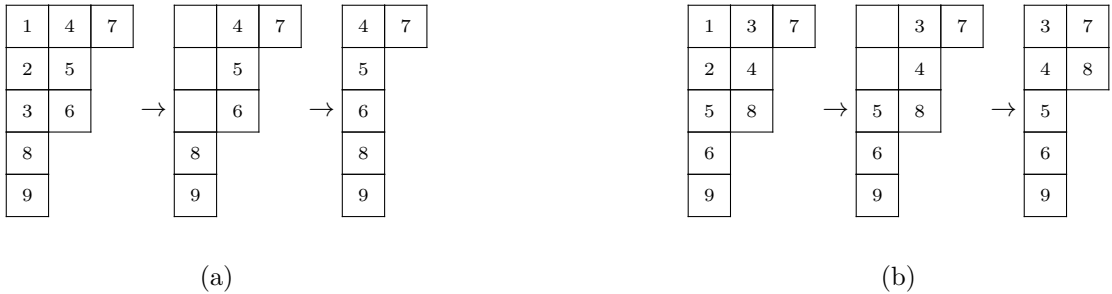

\begin{centering}

\begin{multicols}{2}
$\scriptsize{\tableau{1 & 4 & 7 \\ 2 & 5 \\ 3 & 6 \\ 8 \\ 9} \to \tableau{ & 4 & 7 \\ & 5 \\ & 6 \\ 8 \\ 9} \to\tableau{4 & 7   \\ 5  \\ 6  \\ 8 \\ 9}}$

$\scriptsize{\tableau{1 & 3 & 7 \\ 2 & 4 \\ 5 & 8 \\ 6 \\ 9} \to  \tableau{ & 3 & 7 \\  & 4 \\ 5 & 8 \\ 6 \\ 9} \to \tableau{3 & 7 \\ 4 & 8 \\ 5\\ 6\\9}}$

\end{multicols}

\begin{multicols}{2}
(a)

(b)

\end{multicols}

%
%(a) $\scriptsize{\tableau{1 & 4 & 7 \\ 2 & 5 \\ 3 & 6 \\ 8 \\ 9} \to \tableau{ & 4 & 7 \\ & 5 \\ & 6 \\ 8 \\ 9} \to \tableau{ & 4 & 7 \\ & 5 \\ 6&  \\ 8 \\ 9} \to \tableau{ & 4 & 7 \\ & 5 \\ 6 \\ 8 \\ 9} \to \tableau{ & 4 & 7 \\ 5 & \\ 6  \\ 8 \\ 9}\to }$ \\ \smallskip $\scriptsize{   \tableau{ & 4 & 7 \\ 5  \\ 6  \\ 8 \\ 9} \to  \tableau{4 &  & 7 \\ 5  \\ 6  \\ 8 \\ 9}  \to  \tableau{4 & 7 &  \\ 5  \\ 6  \\ 8 \\ 9}\to  \tableau{4 & 7   \\ 5  \\ 6  \\ 8 \\ 9} \to  \tableau{1 & 4   \\ 2  \\ 3  \\ 5 \\ 6}}$
%
%\medskip
%
%$\tableau{1 & 3 & 7 \\ 2 & 4 \\ 5 & 8 \\ 6 \\ 9} \to \tableau{ & 3 & 7 \\  & 4 \\ 5 & 8 \\ 6 \\ 9} \to \tableau{ & 3 & 7 \\  4&  \\ 5 & 8 \\ 6 \\ 9} \to  \tableau{ & 3 & 7 \\  4&8  \\ 5 &  \\ 6 \\ 9} \to \tableau{ & 3 & 7 \\  4&8  \\ 5   \\ 6 \\ 9} \to \tableau{ & 3 & 7 \\  4&8  \\ 5   \\ 6 \\ 9} \to $

\end{centering}

\caption{Example of correspondence from rectification}
\label{exrect}
\end{figure}
\end{example}

To give a count of multiplicities in terms of avoiding permutations, we need to use the well-known RSK correspondence (c.f. \cite{Sta}), which associates each permutation of $[\ell] = \{1, 2, \ldots, \ell \}$ with an ordered pair $(M,N)$ of standard Young tableaux of the same shape and $\ell$ boxes.  This correspondence works as follows.  We begin with a permutation $w = w_1w_2\cdots w_\ell$ of $[\ell]$ and an ordered pair of empty standard Young tableaux, $(M = \emptyset, N = \emptyset)$.  First, we set $\left (M = \scriptsize{\tableau{w_1}}, N = \scriptsize{\tableau{1}} \right)$.  Then we insert $w_2, w_3, \ldots w_\ell$ in order into $M$ (and $N$) using Schensted's insertion algorithm (see \cite{Sc}, \cite{Sta}), as follows.  At each step, we insert $w_i$ into the first row of $M$ in the spot it should go in increasing order, either replacing a box with a larger entry (say $w_j$) or appearing at the end of the row.  If a box was replaced, that box $(w_j)$ moves to the next row where it belongs in increasing order and either replaces a box or appears at the end of a row.  For each $w_i$, this process continues until some box is added to the end of a row.  At each step, as the insertion algorithm is completed for each $w_i$, a box with entry $i$ is added to $N$ in the position for which the new box was added to $M$.  It is known that $M$ and $N$ are standard Young tableau of the same shape and that given $(M,N)$, we can recover $w$ \cite[Lemma 3]{Sc}.  

A $(k+1)k\ldots21$-avoiding permutation of $[\ell]$ is a permutation with no decreasing subsequence of length $k+1$.  It is known that  using the RSK correspondence (c.f. \cite{Sta}, Corollary 7.23.12)  the number of ordered pairs of standard Young tableaux with the same shape $\mu \vdash \ell$ and less than or equal to $k$ rows is the same as the number of $(k+1)k(k-1)\cdots21$-avoiding permutations of $[\ell]$.

\begin{theorem}\label{avoidingpermTh}
For $k \geq a+b, \max\{a,b\} \leq \ell \leq  \left \lfloor \frac{n+a+b}{2} \right \rfloor-1$, the multiplicity of the weight $k\Lambda_0 - \lambda^\ell_{a,b}$ in $V(k\Lambda_0)$ is equal to the number of $(k+1)k(k-1)\cdots21$-avoiding permutations of $[\ell]=\{1, 2, \ldots, \ell\}$ in which the subsequence of integers 1 through $a$ is in decreasing order and the first $b$ elements are in decreasing order.
\end{theorem}

\begin{proof}
Let $w = w_1w_2\cdots w_\ell$ be a $(k+1)k(k-1)\cdots21$-avoiding permutation of $[\ell]$ such that the subsequence of integers 1 through $a$ is in decreasing order and $w_1 > w_2 > \ldots > w_b$.  Applying the RSK correspondence we obtain a pair of Young tableaux $(M,N)$ of shape 
$\mu \vdash \ell$ with less than or equal to $k$ rows in each.

Observe that since $1, 2, \ldots, a$ occur in decreasing order in the permutation $w$, these entries will be inserted in decreasing order into $M$ via Schensted's algorithm.  When each entry $a, a-1, \ldots, 1$ is inserted into $M$, it will be the smallest entry inserted so far.  Thus, will be inserted into the first row and first column of $M$.  So, at its insertion, $a$ will be in the first row and first column of $M$.  Then at the step when $a-1$ is inserted, $a-1$ will replace $a$ in this position, pushing $a$ to the second row;  since $a$ is smaller than all entries in the second row, $a$ will be inserted in its first column.  Subsequently, at the step when $a-2$ is inserted into $M$, $a-2$ will replace $a-1$ in the first row, which will replace $a$ in the second row, which will move to the first column of the third row.  Continuing in this manner, the entries 1 through $a$ will appear in the first column of $M$.  Thus $M \in Q_a^\mu$.

%Observe that since $1, 2, \cdots , a$ occur in decreasing order in the permutation $w$, while inserting them into the tableaux $M$ using the Schensted's insertion algorithm the first column of $M$ will include these entries (see in Figure \ref{RSKproof}(a)), hence $M$ is a standard Young tableau in $Q_a^\mu$. 
%Notice that when $a$ is inserted into $M$, it will be the smallest entry and so will be in the upper left corner.  The entry $a$ will stay in the upper left corner until $a-1$ is inserted, pushing $a$ to the first column of the second row (since $a-1$ and $a$ are the two smallest entries in $M$).  Continuing in this manner, as $a-2, a-3, \ldots, 1$ are inserted, each entry $a-i$ will appear in row 1.  Since all entries not in the first column will be larger than $a$, the entry $a-i$ will push the entries $a-i+1, a-i+2, \ldots a$ down the first column (for example see in Figure \ref{RSKproof}(a)).    Thus $M$ is a standard Young tableau in $Q_a^\mu$.  

%Now we consider the first $b$ entries of $w$: $w_1 > w_2 > \cdots > w_b$.  When each of these is inserted in $M$ in order, each subsequent $w_i$ will push the previous entries down.  When $w_b$ is inserted, $(M,N)$ will be as in Figure \ref{RSKproof}(b). As $w_i$'s are inserted into $M$ for $i > b$, entries in $N$ do not change position.  Thus, $N$ is a standard Young tableau in $Q_b^\mu$.   Therefore, by Theorem \ref{pairsforRSKLemma}  the result holds.  

Now we consider the first $b$ entries of $w$: $w_1 > w_2 > \ldots > w_b$.  At the conclusion of the first step of Schensted's algorithm, $(M,N)$ will be equal to $\left (\scriptsize{\tableau{w_1}}, \scriptsize{\tableau{1}}\right )$.  At the second step, $w_2$ will be inserted.  Since $w_2 < w_1$, $w_2$ will push $w_1$ down to the first column of the second row of $M$.  Subsequently, $N$ will have a box with $2$ added to the first column of the second row.  Next, $w_3$ will push down $w_2$ and $w_1$ and add a box with a 3 to the first column of the third row of $N$.  Continuing in this manner, at the conclusion of the insertion of $w_b$, the $b^{th}$ step of Schensted's algorithm, $(M,N)$ will be as in Figure \ref{RSKproof}.  As the remainder of the permutation is inserted ($w_{b+1}$ through $w_\ell$), the entries $w_1, w_2, \ldots, w_b$ may not remain in the first column of $M$.  However, since $N$ just has boxes added to it and not shifted, the entries 1 through $b$ will remain in order in the first column.  Thus, at the conclusion of Schensted's algorithm, $N$ will be an element of $Q_b^\mu$. Therefore, by Theorem \ref{pairsforRSKLemma}  the result holds.

\begin{figure}[h]
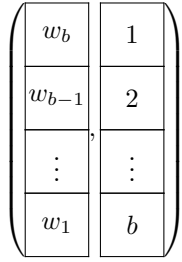

\begin{centering}
%\begin{multicols}{2}
%\scalebox{.7}{$\tableau[l]{1  & \fr[t] &\fr[t]  & \fr[t] &\fr[rtb] \\2 & \fr[] & \fr[] & \fr[] \fr[r]\\ \vdots& \fr[] & \fr[] & \fr[] \fr[rb] \\ a & \fr[] & \fr[r] \\ \fr[l] & \fr[] & \fr[rb] \\ 
%\fr[lb] & \fr[br]  }$}

%\scalebox{.7}{$\tableau[l]{1 \\ \scalebox{.9}{$2$} \\ \vdots \\ a }$}

$\left ( \tableau{w_b \\ w_{b-1} \\ \vdots \\ w_1 }, \tableau{1 \\ 2 \\ \vdots \\ b} \right )$

%\end{multicols}

%\begin{multicols}{2}

%(a)
%
%(b)
%
%\end{multicols}
\end{centering}

%\caption{(a) $M$ after 1 inserted and (b) $(M,N)$ after $w_b$ inserted}
\caption{$(M,N)$ after $w_b$ inserted}
\label{RSKproof}
\end{figure}
\end{proof}

%The first element of the ordered pair is the result of building a standard Young tableau by inserting each number in the permutation, in permutation order, using Schensted's bumping algorithm.  The second element of the ordered pair keeps track of the order in which boxes were added to the first element.  For a given permutation $\sigma = \sigma_1 \sigma_2 \ldots \sigma_\ell$ associated with pair of standard Young tableaux $(P_\sigma,Q_\sigma)$ via the RSK correspondence, $P_\sigma$ will have the first column include all integers 1 through $a$ only when the integers 1 through $a$ appear in decreasing order in $\sigma$, and $Q_\sigma$ will have will have the first column include all integers 1 through $b$ only when $\sigma_1 > \sigma_2 > \cdots > \sigma_b$.     With these considerations and Lemma \ref{orderedpairSYT} we obtain  the corollary below.  

%

\begin{example}  Let $w = 329861754$.  Note that $w$ is a permutation of $[9]$ with no decreasing subsequence of length 6, the numbers 1, 2, and 3 in decreasing order, and the first two values in decreasing order.  We show that $w$ corresponds to $(M,N)$ in the RSK correspondence.

\begin{figure}[h]
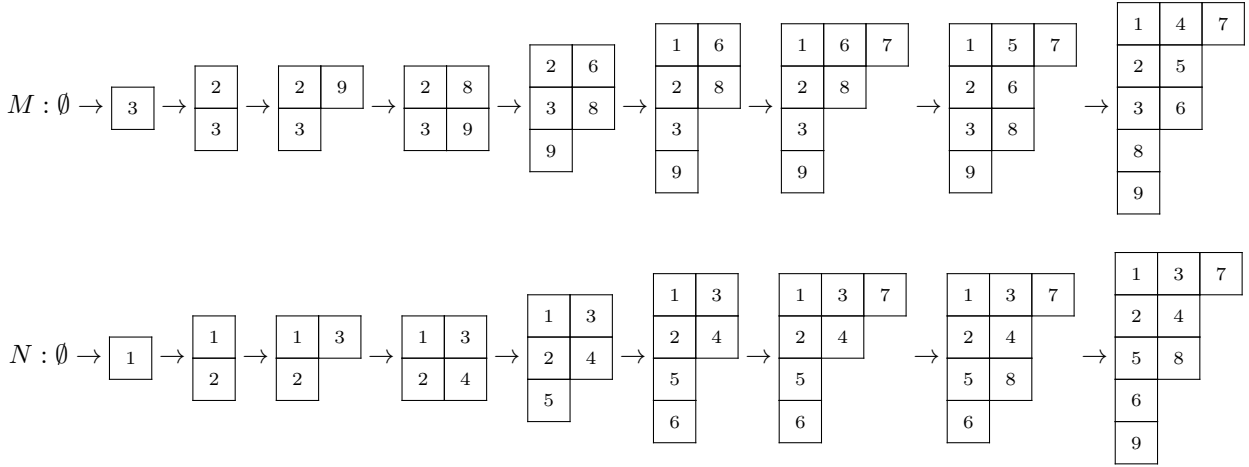

\begin{centering}

$$M: \emptyset \to \scriptsize{\tableau{3}} \to \scriptsize{\tableau{2\\ 3}} \to \scriptsize{\tableau{2 & 9 \\ 3}} \to \scriptsize{\tableau{2 & 8 \\ 3 & 9 }} \to \scriptsize{\tableau{2 & 6 \\ 3 & 8 \\ 9 }} \to \scriptsize{\tableau{1 & 6 \\ 2 & 8 \\ 3 \\ 9 }} \to \scriptsize{\tableau{1 & 6 & 7 \\ 2 & 8 \\ 3 \\ 9 }} \to \scriptsize{\tableau{1 & 5 & 7 \\ 2 & 6 \\ 3 & 8 \\ 9 }}  \to \scriptsize{\tableau{1 & 4 & 7 \\ 2 & 5 \\ 3 & 6 \\ 8 \\ 9 }} $$

$$N: \emptyset \to \scriptsize{\tableau{1}} \to \scriptsize{\tableau{1\\ 2}} \to \scriptsize{\tableau{1& 3 \\ 2}} \to \scriptsize{\tableau{1 & 3 \\ 2 & 4 }} \to \scriptsize{\tableau{1 & 3 \\ 2 & 4 \\ 5 }} \to \scriptsize{\tableau{1 & 3 \\ 2 & 4 \\ 5\\ 6 }} \to \scriptsize{\tableau{1 & 3 & 7 \\ 2 & 4 \\ 5 \\6 }} \to \scriptsize{\tableau{1 & 3 & 7 \\ 2 & 4 \\ 5 & 8 \\ 6 }}  \to \scriptsize{\tableau{1 & 3 & 7 \\ 2 & 4 \\ 5 & 8 \\ 6 \\ 9 }} $$

\end{centering}

\caption{RSK correspondence with $w=329861754$}
\label{RSKex}
\end{figure}
\end{example}
\end{section}

\end{document}